\documentclass[12pt, a4paper, twosided, reqno]{amsart}
\usepackage{enumerate, cite}
\usepackage{hyperref}

\newtheorem{theorem}{Theorem}
\newtheorem{lemma}{Lemma}

\newtheorem{example}{Example}

\allowdisplaybreaks
\author[N. Gahlian and  S. Kumar ]{Nidhi Gahlian and Sanjay Kumar }

\address{nidhi gahlian; department of mathematics, university of delhi, delhi-110007, india.}
\email{nidhigahlyan81@gmail.com}
\address{Sanjay Kumar Pant; department of mathematics, deen dayal upadhyaya college, university of delhi, new delhi-110078, india.}
\email{skpant@ddu.du.ac.in}

\thanks {Research work of the first author is supported by research fellowship from Department of Science and Technology(INSPIRE), New Delhi, India, IF-190674.}
\title[Meromorphic Solutions]{Meromorphic Solutions of Certain Kind of  Non-Linear Differential Equation with Finite Sums of Exponential functions}
\begin{document}
	\maketitle
	
	\begin{abstract}
	We study the existence and nonexistence of meromorphic solutions of nonlinear differential equations involving differential polynomials and finite sums of exponential-type functions with entire coefficients of controlled growth. Employing Nevanlinna theory, we obtain structural results showing that any admissible meromorphic solution must reduce to an exponential-type function governed by the dominant exponential term. The obtained lemma and main result strengthen several existing results in nonlinear complex differential equations with exponential terms.
	
	\end{abstract}
	\section{\textbf{Introduction}}

	The study of nonlinear complex differential equations and their meromorphic solutions has been an active area of research in value distribution theory. A central theme in this direction is to understand how the growth of solutions is influenced by the structure of the differential equation. More broadly, the theory of differential equations seeks to determine the qualitative nature of solutions, including their existence, non-existence, growth properties, order of growth, and possible forms. 
	 An important extension of the Tumura–Clunie theorem was given by Hayman \cite{hay} in $1964$, which has significantly contributed to the study of differential polynomials in meromorphic functions.	
	Hayman \cite{hay} considered the following  non linear differential
	difference equations\begin{equation}\label{ha}
		f^n +Q_d(f(z)) = g(z),
	\end{equation}
	where, $Q_d(f(z))$ is a  differential polynomial
	in $f(z)$ with degree $d$ and 
	obtained the following result.
	\begin{theorem}\label{a}\cite{hay}
		 If non-constant meromorphic functions $f (z)$ and $g(z)$ satisfy
		$N(r, f)+ N(r,\frac{1}{g}) = S(r, f)$ and $d \leq n-1$ in \eqref{ha}, then $g(z)=(f(z)+\gamma(z))^n$,
		where $\gamma(z)$ is a meromorphic function and a small function of $f (z)$.
	\end{theorem}
	 Theroem \ref{a} is an extension of Tumura-Clunie theory, which originated  from a  theorem proposed by Tumura and subsequently proved  by J. Clunie. Since then, nonlinear differential equations have attracted considerable attention, leading to numerous extensions and refinements of the original theory. 
	   Early studies focused on the case $g(z)=u(z)e^{\nu(z)}$, where $\nu(z)$ and $u(z)$ are polynomials. 
	 	In $2004$, Yang and Li \cite{yangcc} investigated the differential equation
	 \begin{equation*}
	 	4f^3+3f''=-\sin3z
	 \end{equation*} and proved that it admits exactly three non-constant entire solutions. Since the function $-\sin3z$ can be expressed as a linear combination of exponential functions, its right-hand side is of exponential type. This observation naturally  motivated the study of nonlinear differential equations whose right-hand sides are finite sums of exponential terms. Beginning with right hand sides of the type
	 $p_1(z)e^{\alpha z}+p_2(z)e^{-\alpha z}$, the theory was gradually extended to  $g(z)=p_1(z)e^{\alpha_1 z}+p_2(z)e^{\alpha_2 z},$ and eventually to finite sums of exponential terms with polynomial coefficients. one can see \cite{bo,chen,ch,wen,ccy}. \\A natural extension of these studies is to ask what happens  when the  dominant term $f^n$ is replaced by $f^nf'$? The study of nonlinear differential equations involving the term $f^nf'$ originates from the theory of differential polynomials in the work of Hayman. The modern form 
	\begin{equation}\label{x}
		f^{n}f^{'}+Q_{d}(z,f)=u(z)e^{\nu(z)},
	\end{equation}  
	was investigated in detail in recent decades in connection with meromorphic solutions of nonlinear differential equations, such as the order of growth, and the characterization of possible forms of solutions. 
		In $2014$ Liao and Ye \cite{jl} systematically investigated  the nonlinear differential equation \eqref{x} and established the following theorem.
	\begin{theorem}\label{2a}\cite{jl} Let $Q_d(z,f)$ be a differential polynomial in $f(z)$ of degree $d$ with rational function coefficients. Suppose that   $u(z)$ is a non-zero rational function, $\nu(z)$ is a non constant polynomial.If $d \leq n-1$, and the differential equation \eqref{x}	admits a meromorphic solution  with finitely many poles, then $f(z)$ has the following form:
		$$Q_{d}(z,f)\equiv0, \qquad f(z)=r(z)e^{\nu(z)/n+1},$$  
	where	$r(z)$ is a rational function satisfying $r^{n}((n+1)r'+\nu^{'}r)=(n+1)u$.
	\end{theorem}
	Theorem \ref{2a} extends an earlier result of Hayman \cite{wk} and has inspired extensive subsequent research in this direction. In particular, considerable attention has been given to nonlinear differential equations whose right-hand sides are exponential polynomials. one can see \cite{zh,jf} and the references therein.

	
	Recently, in $2023$, Chen and Feng \cite{jf} posed the following question.\\
	\textbf{Question:} What can be said for the solutions if the dominant term of the equation \eqref{ha}  is changed from $f^n$ to $f^nf'$, and the right side of the	equation \eqref{ha} contains more than two exponential terms$?$\\
	They subsequently provided an affirmative answer to this question and established the following theorem.
	

	\begin{theorem}\label{hb}\cite{jf}
Let $n\geq 3$, $t\geq 0$ and $m\geq 1$ be integers, $n\geq m$, and $P(z,f,f',...,f^{(t)})$ be differentail polynomial in $f(z)$ of degree $d\leq n$ with small functions of $f(z)$ as its coefficients. Suppose that $P_i$ and $\alpha_i$ are nonzero constants for $i=1,2,...,m,$ and $|\alpha_1|>|\alpha_2|>...>|\alpha_m|$. If $f(z)$ is a meromorphic solution of the differential equation
\begin{equation}\label{3a}
	f^nf'+P(z,f,f',...,f^{(t)}) = P_1e^{\alpha_1z}+P_2e^{\alpha_2z}+...+ P_me^{\alpha_m z},
\end{equation}
then $f(z)=q_1e^{\frac{\alpha_z}{n+1}}$, where $q_1$ is a nonzero constant such that $q_1^{n+1}=\frac{(n+1)P_1}{\alpha_1}$, and $\alpha_1, \alpha_2, ..., \alpha_m $ are in one line.
	\end{theorem}
However, we observed  theorem \ref{hb} contains a gap, and its conclusion is contradicted by following  example.
\begin{example}\label{re} The differential equation
		$$f^3f'+f^2+1=2e^{8z}+6e^{6z}+5e^{4z}$$ \text{has solution} $f(z)=e^{2z}+1$.
		Here, $n=3, m=3, \; \text{and} \; d=2.$
\end{example}
This example \ref{re} satisfies all the assumptions of the theorem; however, its solution is not of the form described in the conclusion of the theorem.
\begin{example}
	The function $f(z)=e^z+z$ is a solution of the following differential equation
	\begin{equation}\label{re3}
		f^3f'+z^2f+z^3f'-z^3=e^{4z}+(z+1)e^{3z}+(z^2+z)e^{2z},
	\end{equation}
	here $n=3$, $d=1$, and $m=3$.
\end{example}
In Example $2$, the solution does not belong to the class described in Theorem \ref{hb}, even though all the hypotheses of the theorem are satisfied. Nevertheless, the examples examined indicate a recurring structure in the solutions. In particular, they often take the form $f(z)=q(z)e^{\frac{\alpha_z}{n+1}} +p(z)$, where $p(z)$ is any polynomial. This observation suggests that functions of this form may naturally arise as solutions under the conditions of Theorem 1.
If the condition $|\alpha_1|>|\alpha_2|>...>|\alpha_m|$ is relaxed, then  solutions of different form may arise, as illustrated by following example.
\begin{example}
	The function $f(z)=e^z+e^{-z}$ is a solution of the following differential equation
	$$f^3f'-2f''f+4=e^{4z}-e^{-4z}-4e^{-2z}.$$
\end{example} 

A natural question arises:\\
\textbf{Question 1:} What can be said for the solutions if the  right side of the equation \eqref{3a}  contains more than two exponential terms with order $q$?\\
A key contribution of this paper is a new lemma, developed in Section 2, concerning the nonexistence of finite-order meromorphic solutions of certain exponential-type differential equations. This lemma plays a pivotal role in controlling the growth behaviour of possible solutions.
 \\
	Motivated by the above discussions, we establish the following result, which bridges the gap in Theorem \ref{hb}, improves it, and also provides an affirmative answer to Question 1.
	
	\begin{theorem}
		Let $m \geq 1$, $n \geq m+2$, $t \geq 0$, and $q \geq 1$ be integers, and a differential polynomial	$P\left(z,f,f',\ldots,f^{(t)}\right)$ in $f(z)$ of degree $d \leq n-m-1$ whose coefficients are small functions of $f(z)$. Let $w_i$ $(i=0,1,\ldots,m)$ be nonzero constants such that 
		$|w_1|>|w_2|>...>|w_m|$ and  $B_i(z)$ $(i=0,1,\ldots,m)$ be entire functions of order less than $q$, with $B_i(z)\not\equiv 0, \; i=1,2,\ldots,m$. If $f(z)$ is a meromorphic solution of the differential equation
		\begin{equation}\label{eq:main}
			f^n(z)f'(z)+P\bigl(z,f,f',f'',\ldots,f^{(t)}\bigr)
			=
			B_0(z)+B_1(z)e^{w_1 z^q}
			+\cdots+
			B_m(z)e^{w_m z^q}
		\end{equation}
		satisfying
		$N(r,f)=S(r,f)$
		then $f(z)=q(z)e^{\frac{w_1}{n+1}z^q},$
		where $q(z)$ is an entire function satisfying $q(z)=((n+1)K)^{\frac{1}{n+1}}$ with $\rho(K)<q$, $\int B_1(z) e^{w_1z^q}dz=e^{w_1z^q}K(z)$, and $w_1,w_2,\ldots,w_m$ are in one line.
	\end{theorem}
	
	The following examples demonstrates the applicability of the theorem.
	\begin{example}
	 	For $n=3$ and $m=1$, the differential equation
		\begin{equation*}
			f^3f'+f'-2zf+6z^5+z= 6z^5+z+2ze^{4z^2},
		\end{equation*}
	admits the meromorphic solution $f(z)=e^{z^2}$. Here $q=2$, $d=1$, $P(z,f,...,f^{(t)})=f'-2zf+6z^5+z$,$B_0=6z^5+z$, $B_1=2z$, \text{and} $w_1=4$. 
	\end{example}

	The solution in the above example  is of the form of  $f(z)=q(z)e^{\frac{w_1}{n+1}z^q}$. Indeed,  $$\int 2z e^{4z^2}dz=\frac{1}{4}e^{4z^2},$$ which implies that $K(z)=\frac{1}{4}$. Consequently,  $q(z)=(4*\frac{1}{4})^{\frac{1}{4}}=1$. Hence, $f(z)=q(z)e^{\frac{w_1}{n+1}z^q}=e^{z^2}$.

	\begin{example}
		For $n=5$ and $m=2$, the differential equation
		\begin{equation*}
			f^5f'+f^2-(f')^2+z^2= z^2+(z^6+z^5)e^{6z}-(1+z^2)e^{2z},
		\end{equation*}
		admits the meromorphic solution $f(z)=ze^z$. Here $q=1$, $d=2$, $P(z,f,...,f^{(t)})=f^2-(f')^2+z^2$, $B_0=z^2$, $B_1=z^6+z^5$, $w_1=6$, $B_2=1+z^2$, \text{and} $w_2=2$. 
	\end{example}
	
	 Similarly, $$\int (z^6+z^5) e^{6z}dz=\frac{z^6}{6}e^{6z},$$ so that $K(z)=\frac{z^6}{6}$. Therefore, $q(z)=\left(6*\frac{z^6}{6}\right)^\frac{1}{6}=z$. Hence, $f(z)=q(z)e^{\frac{w_1z^q}{n+1}}=ze^z$.
		We have intuition that the condition $d \leq n-m-1$ in Theorem $4$ might be weakened to $d\leq n-m$. The following example shows that, in general, the bound cannot be improved beyond this. Indeed, when $m=3,\; n=5, \; \text{and} \; d=3=n-m+1$, all the remaining assumptions of Theorem $2$ are satisfied, but the conclusion fails.
	
	\begin{example} The function $f(z)=e^z+1$ is a solution of the following differential equation
		$$f^5f'-10f^3+25ff'+4f=-6+e^{6z}+5e^{5z}+10e^{4z}.$$
		\end{example}
	 This example $5$ satisfies all the assumptions of Theorem $2$ except that $d=3$, however, the solution is not of the form $f(z)=q(z)e^{\frac{w_1z^q}{n+1}}$.\\


	Throughout this paper, we assume that the reader is acquainted with the standard notations of Nevanlinna theory. For a meromorphic function $f$, $n(r,f)$, $N(r,f)$, $\Bar{N}(r,f) $, $m(r,f)$ and  $T(r,f)$ denote un-integrated counting function, integrated counting function, reduced counting function, proximity function and characteristic function respectively, see \cite{hay,ilpo,yanglo}.
	
	In section $2$ we state  some definitions, lemmas  and results. In section 3, we prove lemma and  theorem.
	\medskip
	
	\section{\textbf{Auxiliary results}}
	To maintain self-containment, we briefly recall the definitions of the order of growth 
	$\rho(f)$,  hyper-order of growth $\rho_{2}(f)$  and the exponent of convergence of zeros $\lambda(f)$ for a meromorphic function $f(z)$.
	
	$$\rho(f)=\limsup_{r\to\infty}\frac{\log T(r,f)}{\log r},$$
	$$\rho_2(f)=\limsup_{r\to\infty}\frac{\log \log T(r,f)}{\log r},$$
	and
	$$\lambda(f)=\limsup_{r\to\infty}\frac{\log n\left(r,\frac{1}{f}\right)}{\log r}.$$ 
	
	\bigskip
	
	For a meromorphic function $f$, Nevanlinna’s First Main Theorem asserts  that
	$$T\left(r,\frac{1}{f-a}\right)=T(r,f)+O(1),$$
	for all $a\in \mathbb{C}$,
	where $	O(1)$ denotes a bounded quantity depending only on $a$.
	\smallskip
	
	A meromorphic function $g(z)$ is a small function of $f(z)$ if $T(r,g)=S(r,f)$ and vice versa.
	For a meromorphic function $f(z)$,  $S(r,f)$ denotes the quantity satisfying $S(r,f)=o(T(r,f))$, as $r\to\infty$, outside  of a possible exceptional set $E$ (not necessarily same at each occurrence) of  finite linear measure. \\
	A differential polynomial $P$  is an expression formed by finite sum of monomials  involving $f$ and its derivatives. It has the general form
	\begin{equation*}
		P(z,f,f',...,f^{(t)})=\sum_{\lambda \in I}a_{\lambda}(z)f^{\lambda_0}(z)(f')^{\lambda_1}(z)...(f^{(t)})^{\lambda_t},
		\end{equation*}
		where $t$ is a fixed non negative integer, $ I$ is a finite 
		 set of multi-indices $\lambda:=(\lambda_0,\lambda_1,...,\lambda_t)$, $\lambda_j \in \mathbb{N}$ for each $j \in \{1,2,...,t\}$, and the coefficients $a_\lambda(z)$ are small functions of $f(z)$. For each term corresponding to $\lambda$, we define its weight (or total degree) as
		$d_\lambda:=\lambda_0+\lambda_1+...+\lambda_t.$ The degree of the differential polynomial $P$ is then defined 
	 as the largest such weight among all terms appearing in the sum: 	$$\deg P= \max_{\lambda \in I}(\lambda_0+\lambda_1+...+\lambda_t).$$	Borel’s lemma is a fundamental tool in applying Nevanlinna theory to complex differential equations.
	\begin{lemma}\label{imple1}(Borel's Lemma)\cite{cc}
		Suppose $f_1(z),f_2(z),...,f_n(z)(n\geq2)$ are meromorphic functions and $h_1(z),h_2(z),h_3(z),...,h_n(z)$ are entire functions satisfying:
		\begin{enumerate}
			\item $\sum_{i=1}^{\infty} f_i(z)e^{h_i(z)}\equiv 0$.
			\item For $1\leq i< k\leq n$, $h_i-h_k$ are  constants.
			\item  $1\leq i\leq n$, $1\leq m< k\leq n$,  $T(r,f_i(z))=o(T(r,e^{h_m-h_k}))$ as $r\rightarrow \infty$ outside of a set of finite linear measure.
		\end{enumerate}
		Then $f_i\equiv0\ (i=1,2,3...n)$.
	\end{lemma}
	We now state a lemma that gives an estimate for the proximity function of the logarithmic derivative of a meromorphic function $f(z)$. 
	
	\begin{lemma}\label{imple2}\cite{ilpo}
		Suppose $f(z)$ is a transcendental meromorphic function and $k\geq1$ is an integer. Then 
		\begin{equation*}
			m\left(r,\frac{f^{(k)}}{f}\right)= S(r,f).
		\end{equation*}
		If $f$ is of finite order growth, then
		$$m\left(r,\frac{f^{(k)}}{f}\right)=O(\log r).$$
	\end{lemma}

\begin{lemma}\label{imple0}\cite{cc}
	Let $f(z)$ be a nonconstant meromorphic function in $\mathbb{C}$. If $0$, and $\infty$ are Picard exceptional values of $f(z)$, then $f(z)=e^{h(z)}$, where $h(z)$ is a nonconstant entire function. Especially, let $\lambda$, and $\mu$ be the order and the lower order of $f(z)$, respectively, $\mu< \infty$. Then, $\mu$ is a positive integer, $h(z)$ is a polynomial with degree $\mu$, and $\lambda=\mu$. 
\end{lemma}

	Next lemma plays an important part in the study of complex differential-difference
	equations.
	\begin{lemma}\label{imple 4}\cite{yl}
		Let $f(z)$ be  a transcendental meromorphic function, and $P(z,f)$, $Q(z,f)$ be two differential-difference polynomials of $f(z)$. If 
		\begin{equation*}
			f^n(z) P(z,f)=Q(z,f)
		\end{equation*}
		holds, and if the total degree of $Q(z,f)$ in  $f(z)$ and its derivatives and their shifts is atmost $n$, then
		\begin{equation*} 
			m(r,P(z.f))=S(r,f),
		\end{equation*}
		for all $r$ outside of a possible exceptional set of finite logarithmic measure.
	\end{lemma}
	
	Next lemma estimates the characteristic function of an exponential polynomial $f$. This lemma can be seen in \cite{whl}.
	\begin{lemma}\label{imple5}\cite{wen}
		Suppose $f$ is an entire function given by
		$$f(z)=B_{0}(z)+B_{1}(z)e^{w_{1}z^{t}}+B_{2}(z)e^{w_{2}z^{t}}+...+B_{m}(z)e^{w_{m}z^{t}},$$
		where $B_{i}(z);0\leq i\leq m$ denote either exponential polynomial of degree $<t$ or polynomial in $z$, $w_{i};1\leq i\leq m$ denote the constants and $t$ denotes a natural number. Then
		$$T(r,f)=C(Co(W_{0}))\frac{r^{t}}{2\pi}+o(r^{t}),$$
		Here $C(Co(W_{0}))$ is the perimeter of the convex hull of the set $W_{0}=\{0,\overline{w}_{1},\overline{w}_{2},...,\overline{w}_{m}\}$.
	\end{lemma}
		%
	
	
Using ideas from lemma  of \cite[lemma 2.1]{bo}, we prove the follwing lemma which plays a key role in proving our main theorem.
\begin{lemma}
	Let $m,n,q$ be positive integers, $n \geq m \geq 2$, $w_j (j=1,2,3,...,m)$ be distinct constants and $B_j$ are nonzero entire functions  of order less than $q$, then this  equation
\begin{equation}\label{3p}
	f^nf'=B_1(z)e^{w_1z^q}+B_2(z)e^{w_2z^q}+...+B_m(z)e^{w_mz^q}
\end{equation}
 has no meromorphic solutions with $\lambda(f)<\rho(f)$.
\end{lemma}
Next, to understand the growth behavior of solutions of algebraic differential equations, we recall the following result concerning the role of dominant terms.
  \begin{lemma}\label{imple7}\cite{wittich}
  	If the algebraic differential equation
  	$$ P(z,f)=0,$$ where $P(z,f)$ is a differential polynomial in $f(z)$ with polynomial coefficients, has only one dominant term, then the equation has no transcendental entire solutions.
  \end{lemma}
  Since linear combinations of differential polynomials and their derivatives arise naturally in our setting, we state the following lemma.
  \begin{lemma}\label{imple10}\cite{jf}
  	Let $k$ be a positive integer, and let $\alpha$, $\beta$ be nonzero constants. If $P(z,f,f',...,f^{(t)})$ be a differential polynomial in $f(z)$ of degree $d$ with small functions of $f(z)$ as its coefficients, then $\alpha P(z,f,f',...,f^(t))+\beta P^{(k)}(z,f,f',,,f^{(t)})$ is a differential polynomial in $f(z)$ of degree atmost $d$.
  	\end{lemma}

	\section{\textbf{Proof of  lemma and main theorem}}
	\begin{proof}[\textbf{\underline{Proof of Lemma 6}}] 
 Let on the contrary $f(z)$ be a meromorphic solution of equation
	\eqref{3p}, 
	We claim that $f(z)$ is an entire function, If not, then
	$f(z)$  has atleast one pole say $z_0$. From equation \eqref{3p}, it follows that 
	$B_1(z)e^{w_1z^q}+B_2(z)e^{w_2z^q}+...+B_m(z)e^{w_mz^q}$  must have  a pole $z_0$.This is impossible, since the function $B_1(z)e^{w_1z^q}+B_2(z)e^{w_2z^q}+...+B_m(z)e^{w_mz^q}$ is entire  and hence admits $\infty$ as a picard exceptional value. Therefore, f must be  an entire function.
	Next, we show  that $f(z)$ is a transcendental entire function. Suppose on the contrary that $f(z)$ is a polynomial , then $f^nf'$ is clearly small function of  $B_1(z)e^{w_1z^q}+B_2(z)e^{w_2z^q}+...+B_m(z)e^{w_mz^q}$. Applying Borel's lemma \ref{imple1} to  equqation  \eqref{3p}, we obtain $B_j=0 (j=1,2,...,m)$, which is absurd as $B_j $'s are nonzero constants. Hence, $f(z)$ cannot be a polynomial. Therefore, $f(z)$ is transcendental entire function.\\
	If $\rho(f)$ is infinite, then it implies that	$B_1(z)e^{w_1z^q}+B_2(z)e^{w_2z^q}+...+B_m(z)e^{w_mz^q}$ is a small function of $f(z)$. As $f(z)$ is  an admissible transcendental entire solution of \eqref{3p}, the differential equation has only one dominant term, namely $f^nf'$. This contradicts  lemma \ref{imple7}. Hence, $f(z)$ must be of finite order. 
With the help of Lemma \ref{imple5}, we have 
	\begin{align*}
	T(r,f^nf') = (n+1)T(r,f)=&T(r,B_1(z)e^{w_1z^q}+B_2(z)e^{w_2z^q}+...+B_m(z)e^{w_mz^q})\\& =O(r^q).
	\end{align*}
Consequently,  $\rho(f)=q$. We now distinguish between the following two cases.\\
\textbf{1.}	Transcendental function has no zeroes. From Lemma \ref{imple0}, we have $f(z)=e^{g(z)}$, where $g(z)$ is a polynomial of degree $\rho(f)$. Equation \eqref{3p} can be rewritten as  
	\begin{equation*}
		g'(z)e^{(n+1)g(z)}=B_1(z)e^{w_1z^q}+B_2(z)e^{w_2z^q}+...+B_m(z)e^{w_mz^q}.
		\end{equation*}
$(a)$ If there does not  exists any $l \in \{1,2,...,m\}$ such that $(n+1)g^{(q)}(z)\neq w_l$, then 
\begin{equation*}
	B_1(z)e^{w_1z^q}+B_2(z)e^{w_2z^q}+...+B_m(z)e^{w_mz^q}- g'(z)e^{(n+1)g(z)}=0.
\end{equation*}  
By Borel's lemma \ref{imple1}, we have $B_i=0$ for all $i \in \{1,2,...,m\}$. This contradicts with the assumption $B_i$ is a nonzero entire function.\\
$(b)$ If there exist some  $l \in \{1,2,...,m\}$ such that $(n+1)g^{(q)}(z)= w_l$,implies $g(z)$ is a polynomial of degree $q$, i.e $g(z)=a_qz^q+...+a_0$, where $a_q(\neq0), a_{q-1},...,a_0$  are constants, then
\begin{equation*}
	B_1(z)e^{w_1z^q}+...+(B_l(z)-\gamma e^{(n+1)a_0})e^{w_lz^q} +...+ B_m(z)e^{w_mz^q}=0.
\end{equation*}
Again by Borel Lemma \ref{imple1}, we obtain $B_i=0$ for all $i \neq l$. This contradicts with the assumption $B_i$ is a nonzero entire function.\\
\textbf{2.}	Transcendental function has at least one zero, say $z_0$ be a zero of $f(z)$ with multiplicity $t \geq 1 $.Then, we get $f(z)=(z-z_0)^t \psi (z)$, where $\psi (z)$ is an analytic function in the neighbourhood of $z_0$ satisfying $\psi (z_0)\neq 0$.\\
As order is finite  and also $\lambda(f)<\rho(f)$, with the help of Hadamard's factorization $f(z)$ can be written as $\beta(z)e^{P(z)}$, where $P(z)$ is a polyomial  of degree $q$ say $P(z)=a_qz^q+q_{q-1}z^{q-1}+...+a_0$,  and $\beta(z)$ is an entire function with order less than $q$, as $\rho(\beta)=\lambda(f)<\rho(f)(a_q\neq0)=q$.  \\
Now substituting $f(z)$ in equation \eqref{3p}, we have 
\begin{align*}
	&(\beta(z)e^{P(z)})^n(\beta(z)e^{P(z)})'=B_1(z)e^{w_1z^q}+B_2(z)e^{w_2z^q}+...+B_m(z)e^{w_mz^q},\\&
	e^{nP}\beta^n(\beta'e^{P}+e^{P}P'\beta)=B_1(z)e^{w_1z^q}+B_2(z)e^{w_2z^q}+...+B_m(z)e^{w_mz^q},\\& \beta^n(\beta'+P'\beta)e^{(n+1)P}-B_1(z)e^{w_1z^q}-B_2(z)e^{w_2z^q}-...-B_m(z)e^{w_mz^q}=0.
\end{align*}
Set\begin{equation*}
	A(z)=\beta^n(\beta'+P'\beta).
\end{equation*}
 $\rho(\beta)<q$ implies $\rho(A)<q$. Above equation can be written as
\begin{equation}\label{2q}
	A(z)e^{(n+1)P(z)}-B_1(z)e^{w_1z^q}-B_2(z)e^{w_2z^q}-...-B_m(z)e^{w_mz^q}=0
\end{equation} we can apply Borel's Lemma \ref{imple1} as rest assumptions are assured. We have two cases:\\
\textbf{(a)} If for every $j \in \{{1,2,...,m}\}$,
\begin{equation*}
(n+1)P(z)-w_jz^q
\end{equation*}
 is nonconstant, then  all exponents are distinct and Borel's lemma yields 
 $$ A(z)\equiv B_1\equiv...B_m\equiv 0,$$ which is a contradiction to the assumptions.\\
 \textbf{(b)} If for some $j \in \{{1,2,...,m}\}$,
 \begin{equation*}
 	(n+1)P(z)-w_jz^q
 \end{equation*}
 is constant, $i.e$ $(n+1)a_q= w_j$, where $P(z)=\frac{w_j}{n+1}z^q+c$, then equation \eqref{2q} becomes\\
 \begin{equation*}
 		A(z)e^c-B_j(z)-\sum_{\substack{k=1\\ k\neq j}}^{m}
 		B_k e^{(w_k-w_j)z^q}=0,
 	\end{equation*}
 all exponenets $0, (w_j-w_k)z^q (k \neq j)$ are all distincts as $w_k's$ are distincts, 	applying Borel's lemma \ref{imple1} yields \\
 \begin{equation*}
 		A(z)e^c-B_j(z)\equiv 0, B_k\equiv0  (k\neq j),
 \end{equation*}
 which is again a contradiction to the assumption that $B_k's$ are nonzero entire functions.
 \end{proof}
 	\begin{proof}[\textbf{\underline{Proof of Theroem 4}}] 
 		
 		Assume that $f(z)$ is a meromorphic solution of equation \eqref{eq:main} with $N(r,f)=S(r,f)$. We first prove that $f(z)$ is transcendental. Suppose to the contrary that $f(z)$ is a rational. Then the quantity
 		$f^nf'+P(z,f,f',...,f^(t))$	is a small function  of $B_0+B_1(z)e^{w_1z^q}+B_2(z)e^{w_2z^q}+...+B_m(z)e^{w_mz^q}$.  Applying lemma \ref{imple1} with the assumption $|w_1|>|w_2|>...>|w_m|$,  we obtain $B_j\equiv0 \; (j=1,2,...,m)$. This contradicts with the hypothesis that $B_j$'s are entire functions. Hence, $f(z)$ must be a transcendental meromorphic function. \\For notational convenience, set $P=P(z,f,f',...,f^(t))$. We next show that $\rho(f)\geq q$. Assume, for contradiction, that $\rho(f)<q$. Then $\rho(f^nf'+P)<q$. By comparing the growth of the characteristic functions on both sides of equation \eqref{eq:main} using lemma \ref{imple2}, we obtain a contradiction. Therefore, $\rho(f)\geq q$. Consequently, each $B_j (j=0,1,…,m)$ is a small function of $f$, because $B_j(z)$ is entire of order strictly less than $q$.
 		
 		First we discuss the case when $B_0(z)\equiv0$. Equation \eqref{eq:main} becomes 
 		\begin{equation}\label{1}
 			f^nf'+P=B_1(z)e^{w_1z^q}+B_2(z)e^{w_2z^q}+...+B_m(z)e^{w_mz^q}.
 		\end{equation}
 		Differentiating both sides of equation \eqref{1}, we get 
 		\begin{equation}\label{2}
 			 nf^{n-1}(f')^2+f^nf''+P'=(B'_1 +B_1w_1qz^{q-1})e^{w_1z^q} +...+(B'_m +B_mw_mqz^{q-1})e^{w_mz^q}.
 		\end{equation}
 		Eliminating $e^{w_1z^q}$ from equations \eqref{1}and \eqref{2}.\\
 		\textbf{Case(A)} $m=1$.
 		\begin{align*}
 			&(B'_1 +B_1w_1qz^{q-1})(f^nf'+P)-B_1( nf^{n-1}(f')^2+f^nf''+P')=0.\\& 
 			\text{ \textbf{(a).} Assume that}\\&
 B_1 (nf^{n-1}(f')^2+f^nf'')	-(B'_1 +B_1w_1qz^{q-1})f^nf'=(B'_1 +B_1w_1qz^{q-1})P-B_1P'\equiv0
 		\end{align*}
 		By some  calculations, we have $f^nf'=P$, and $ P=c_1e^{w_1z^q}B_1$. Hence, 
 		\begin{equation}\label{3}
 			f^nf'=c_1e^{w_1z^q}B_1.
 		\end{equation}
 		On integrating both sides, we have $$\frac{f^{n+1}}{n+1}=c_1e^{w_1z^q}K,$$ where $K(z)$ is an entire function  given by $ \int B_1(z) e^{w_1z^q}dz=e^{w_1z^q}K(z)$ and has order less than $q$. Hence, we have $$f=(CK)^{\frac{1}{n+1}}e^{\frac{w_1z^q}{n+1}},$$ where $C=c_1(n+1)$.
 		Substituting value of $f(z)$ in
 		\begin{align*}
 			P(z,f,f',...,f^{(t)})=& \sum_{\gamma}\alpha_{\gamma}(z)(f)^{i_0}(f')^{i_1}...(f^{(t)})^{i_t}\\&=\sum_{\gamma}\alpha_{\gamma}(z)\left(K_0e^{\frac{w_1z^q}{n+1}}\right)^{j_0}\left(K_1e^{\frac{w_1z^q}{n+1}}\right)^{j_1}...\left(K_te^{\frac{w_1z^q}{n+1}}\right)^{j_t}\\&=\sum_{\gamma}\alpha_{\gamma}(z)K_0^{j_0}K_1^{j_1}K_2^{j_2}...K_t^{j_t}e^{\frac{w_1z^q(j_0+j_1+...+j_t)}{n+1}}\\&=\sum_{\gamma}\alpha_{\gamma}(z)K_\gamma(z)e^{\frac{\gamma w_1z^q}{n+1}},
 			\end{align*}
 			where $$K_0=(CK)^{\frac{1}{n+1}},\qquad \text{and}\qquad K_{p+1}=K_p'+K_p\frac{w_1}{n+1}qz^{q-1},$$for $p=0,1,2,...,t$ such that $$K_{\gamma}=\sum_{j_0+j_1+...+j_t=\gamma}K_0^{j_0}K_1^{j_1}...K_t^{j_t}.$$  Since $K_0,K_1,...,K_t$ are entire functions of order  less than $q$, it follows that $\alpha_{\gamma}(z)$, and $K_{\gamma}(z)$ are small function of $f(z)$. Observe that $$0\leq \gamma=j_0+j_1+...+j_t \leq d\leq n-2(m=1),$$ which immediately yields $\frac{\gamma}{n+1}<1$. Substituting equation \eqref{3} in equation \eqref{eq:main} and using lemma \ref{imple1} 
 			 \begin{align*}
 				&c_1e^{w_1z^q}B_1+P=B_1e^{w_1z^q}\\&P=(1-c_1)B_1e^{w_1z^q},
 			\end{align*}
 			gives $c_1=1$, and $P(z,f,f',...,f^{(t)})\equiv 0$.\\
 			\textbf{(b).} Suppose that
 			\begin{equation*}
 				B_1 (nf^{n-1}(f')^2+f^nf'')	-(B'_1 +B_1w_1qz^{q-1})f^nf'=(B'_1 +B_1w_1qz^{q-1})P-B_1P'\not\equiv0.
 			\end{equation*}
 	Above equation can be written as
 	\begin{equation*}
 		f^{n-1}Q(z)=(B'_1 +B_1w_1qz^{q-1})P-B_1P',
 		\end{equation*}
 		where $Q(z)=nf'^2B_1+ff''B_1-ff'(B_1'+B_1w_1z^{q-1}q)$.
 	By lemma \ref{imple10}, the 	R.H.S of this equation is of atmost degree $d$, and $d\leq n-m-1=n-2(m=1)$. As $n\geq m+2=3$, it follows with the help of lemma \ref{imple 4} that $$m(r,Q)=S(r,f) \qquad m(r,fQ)=S(r,f).$$ Consequently, with the help of  $N(r,f)=S(r,f)$, we obtain $T(r,f)=S(r,f)$ , which is a contradiction.
 	Therfore for $m=1$ and  we have $f(z)=q(z)e^{\frac{w_1z^q}{n+1}}$, where $q(z)=((n+1)K)^{\frac{1}{n+1}}$, and $P\equiv0$.\\
 	\textbf{CaseB.}  $m\geq2$.\\ Set $B_1'+B_1w_1qz^{q-1}=S$, and eliminating $e^{w_1z^q}$,we get
 	\begin{align}\label{9}
 		nB_1f^{n-1}f'^2+&f^nf''B_1-f^nf'S+B_1P'-SP\\&\nonumber=(B_1B_2'-B_1'B_2+(w_2-w_1)B_1B_2qz^{q-1})e^{w_2z^q}+...\\&\nonumber+(B_1B_m'-B_1'B_m+(w_m-w_1)B_1B_mqz^{q-1})e^{w_mz^q}.
 	\end{align}
 	Set $Q_1=nB_1f^{n-1}f'^2+f^nf''B_1-f^nf'S, \quad T_1=B_1P'-SB,$ and
 	$B_{1i}=B_1B_i'-B_1'B_i+(w_i-w_1)B_1B_iqz^{q-1}$.
 	Equation \eqref{9} can be rewritten as 
 	\begin{equation}\label{10}
 		Q_1+T_1=B_{12}e^{w_2z^q}+B_{13}e^{w_3z^q}+...+B_{1m}e^{w_mz^q}.
 	\end{equation}
 	Consider the following subcases:\\
 \textbf{Subcase B(i)} If $Q_1\equiv0 \; \text{and} \; T_1\equiv0$, then by Borel's lemma \ref{imple1}, and equation  \eqref{10}, we have 
 $B_{1i}\equiv0 \ \forall i \in {1,2,...,m}$. Thus,
 \begin{equation*}
 B_1B_i'-B_1'B_i+(w_i-w_1)B_1B_iqz^{q-1}=0.
 \end{equation*}
 On integrating, we have
 $\frac{B_1}{B_i}=ce^{(w_i-w_1)z^q}$, where $c$ is a nonzero constant. Under the assumption that $\rho(B_i)<q \;\; \text{for} \; i=2,..,m$, it follows that $w_i-w_1=0$. However this is impossible because $|w_1|>|w_2|>...>|w_m|$.\\
\textbf{Subcase B(ii)}$Q_1\equiv0$ and $T_1\not \equiv0$.  So $Q_1=nB_1f^{n-1}f'^2+f^nf''B_1-f^nf'S\equiv0$. Now, from case(A) we have $f^nf'=c_1Ke^{w_1z^q}$. Substituting $f(z)$ in equation \eqref{eq:main} and \begin{align*}
		P(z,f,f',...,f^{(t)})=& \sum_{\gamma}\alpha_{\gamma}(z)K_\gamma(z)e^{\frac{\gamma w_1z^q}{n+1}}\\&=(1-c_1)B_1e^{w_1z^q}+B_2e^{w_2z^q}+...+B_me^{w_mz^q},
	\end{align*}
 where $\rho(\alpha_{\gamma})$, $\rho(K_{\gamma})<q$. As $\frac{\gamma}{n+1}<1$, $|w_1|>|w_2|>...>|w_m|$, and $B_i\not\equiv0$ for $j=2,3,...,m$. Applying Borel's lemma \ref{imple1},	we have $c_1=1$ and $\frac{\gamma_j}{n+1}w_1=w_j$, and $\alpha_{\gamma}(z)K_\gamma(z)=B_j$ $j\in \{2,3,...,m\}$ and there exist some $t \in \{2,3,...,m\}$ such that $\gamma_t=d$. Hence $f(z)=q(z)e^{\frac{w_1z^q}{n+1}}$, where $q(z)=((n+1)K)^{\frac{1}{n+1}}$, and $P=B_2e^{w_2z^q}+...+B_me^{w_mz^q}$.\\
 \textbf{Subcase B(iii)} $Q_1\not\equiv0$ and $T_1\equiv0$. Then from the definition of $T_1$, we have $P=c_2e^{w_1z^q}$,where $c_2$ is a nonzero constant. Combine it with equation \eqref{eq:main}, we have 
 \begin{equation*}
 	f^nf'=(1-c_2)B_1e^{w_1z^q}+B_2e^{w_2z^q}+...+B_me^{w_mz^q}.
 \end{equation*}
 Observe that $B_i\not \equiv0$ for $i=1,2,...,m$.\\
 \textbf{(a)} If $m=2$, then for $ c_2\neq1$, in view of the lemma \ref{3} the equation
 \begin{equation*}
 	f^nf'=(1-c_2)B_1e^{w_1z^q}+B_2e^{w_2z^q}
 	\end{equation*}
 	has no meromorphic solution with $N(r,f)=S(r,f)$. For $c_2=1$, we have $f^nf'=B_2e^{w_2z^q}$. This gives $f=(G(n+1))^{\frac{1}{n+1}}e^{\frac{w_2z^q}{n+1}}$, where $G$ is a nonzero entire function with order less than $q$. Substituting $f(z)$ in equation \eqref{eq:main} and $P(z)$, we have 
 	$$P(z,f,f',...,f^{(t)})=\sum_\gamma\alpha_{\gamma}(z)K_{\gamma}(z)e^{\frac{\gamma w_2z^q}{n+1}},$$ and $B_2e^{w_2z^q}+P=B_1e^{w_1z^q}+B_2e^{w_2z^q}$, where order of $K_\gamma$ and $\alpha_\gamma$ are less than $q$. As $\frac{\gamma}{n+1}<1$, and $|w_1|>|w_2|$, so lemma \ref{imple1}  gives $B_1\equiv0$, which is a contradiction.\\
 	\textbf{(b)} If $m\geq3$, then lemma \ref{3p} implies non existence of meromorphic solutions with $\lambda(f)<\rho(f)$ i.e $N(r,f)=S(r,f)$.\\
 	\textbf{Subcase B(iv)} $Q_1\not\equiv0$ and $T_1\not\equiv0$. Differentiating equation \eqref{10}
 	\begin{equation*}
 		Q_1'+T_1'=(B_{12}'+B_{12}w_2qz^{q-1})e^{w_2z^q}+...+(B_{1m}'+B_{1m}w_mqz^{q-1})e^{w_mz^q}
 	\end{equation*}
 	\textbf{(a)} If $m=2$, then eliminating $e^{w_2z^q}$, we have
 	 \begin{align*}
 		&(B_{12}'+B_{12}w_2qz^{q-1})(Q_1+T_1)-B_{12}(Q_1'+T_1')=0.  
 	\end{align*}
 If
 		\begin{equation}\label{11}
 			B_{12}Q_1'-(B_{12}'+B_{12}w_2qz^{q-1})Q_1=(B_{12}'+B_{12}w_2qz^{q-1})T_1-B_{12}T_1'\equiv0,
 		\end{equation} 
 then  from equations \eqref{10} and \eqref{11} according to \cite{ac}, we can solve to get 
 \begin{equation*}
 	f^nf'= c_3B_1e^{w_1z^q}+c_4B_2e^{w_2z^q},
 	\end{equation*}
 		where $c_3,c_4$ are constants and $c_4 \neq0$. Now, if $c_3\neq0$, then  lemma \ref{3} implies nonexistence of mermorphic solutions.
 		If $c_3=0$, then $f^nf'=c_4B_2e^{w_2z^q}$ and proceding as in case $(A)$, we have contradiction using Borel' lemma \ref{imple1}.\\
 	On the other hand if
 		\begin{equation}\label{12}
 			B_{12}Q_1'-(B_{12}'+B_{12}w_2qz^{q-1})Q_1=(B_{12}'+B_{12}w_2qz^{q-1})T_1-B_{12}T_1'\not\equiv0,
 		\end{equation}
 		then by the definition of $B_{1i}$ and lemma \ref{imple10}, the R.H.S of equation \eqref{12} is a differential polynomial in $f$ of atmost degree  $d\leq n-m-1=n-3$. However the L.H.S is a differential polynomial of degree $n+1$, where $n\geq m+2=4$ and contains the term $f^{n-2}$. Hence, we get 
 		\begin{equation}\label{13}
 			f^{n-2}R=B_{12}w_2qz^{q-1})T_1-B_{12}T_1'.
 			\end{equation}
 	Applying Lemma \ref{imple 4} to the equation \eqref{13} implies 
 	$$m(r,R)=S(r,f) \; \text{and} \; m(r,fR)=S(r,f).$$
 	Similar to case$(A)$, we have 
 	\begin{align*}
 		T(r,f)=m(r,f)+S(r,f)=m\left(r,\frac{fR}{R}\right)+S(r,f)=S(r,f).
 	\end{align*}
 	Hence, a contradiction.\\  
 	\textbf{(b)}If $m\geq3$, then eliminating $e^w_2z^q$, we have
 	\begin{align*}
 		&B_{12}Q_1'-(B_{12}'+B_{12}w_2qz^{q-1})Q_1-(B_{12}'+B_{12}w_2qz^{q-1})T_1+B_{12}T_1'=\\&(B_{12}B_{13}' -B_{12}'B_{13}+(w_3-w_2)B_{12}B_{13}qz^{q-1})e^{w_3z^q}+\\&...+
 		(B_{12}B_{1m}' -B_{12}'B_{1m}+(w_m-w_2)B_{12}B_{1m}qz^{q-1})e^{w_mz^q}).
 	\end{align*}
 	Set $Q_2=B_{12}Q_1'-(B_{12}'+B_{12}w_2qz^{q-1})Q_1$, $T_2=B_{12}T_1'-B_{12}w_2qz^{q-1})T_1$, and $B_{2i}=(B_{12}B_{1m}'-B_{12}'B_{1i}+(w_i-w_2)B_{12}B_{1i}qz^{q-1})$.
 	Then equation  can be written as
 	\begin{equation*}
 		Q_2+T_2=B_{23}e^{w_3z^q}+...+B_{2m}e^{w_mz^q}.
 	\end{equation*}
 Now, we consider the four subcases  $Q_2\equiv0 \; \text{and} \; T_2\equiv0$, $Q_2\equiv0  \; \text{and}\;  T_2\not\equiv0$,  $Q_2\not\equiv0 \; \text{and} \; T_2\equiv0$, and $Q_2\not\equiv0 \;\text{and}  \;  T_2\not\equiv0$ as in Case $(B)$ and obtain a contradiction in first three cases and in fourth subcase again differentiating and eliminating $e^{w_3z^q}$, using similar arguments upto $m^{th}$ stage, we have
 \begin{equation}\label{14}
 	Q_{m-1}+T_{m-1}=B_{m-1,m}e^{w_mz^q}.
 	\end{equation}
 	where $Q_{m-1}=B_{m-2,m-2}Q_{m-2}'-(B_{m-2,m-1}'+B_{m-2,m-1}w_{m-1}qz^{q-1})Q_{m-2}$,\\$T_{m-1}=B_{m-2,m-2}T_{m-2}'-(B_{m-2,m-1}'+B_{m-2,m-1}w_{m-1}qz^{q-1})T_{m-2},$ and \\$B_{m-1,m}=(B_{m-2,m-1}B_{m-2,m}'-B_{m-2,m-1}'B_{m-2,m}+(w_m-w_{m-1})B_{m-2,m-1}B_{m-2,m}qz^{q-1})$.
 	\\ Similarly, we define four cases here:\\
 	\textbf{(i)} $Q_{m-1}\equiv0$ and $T_{m-1}\equiv0$.\\
 	Equation \eqref{14} implies \begin{equation*}
 		B_{m-1,m}=(B_{m-2,m-1}B_{m-2,m}'-B_{m-2,m-1}'B_{m-2,m}+(w_m-w_{m-1})B_{m-2,m-1}B_{m-2,m}qz^{q-1})=0,
 	\end{equation*}
 	and integrating it we have 
 	\begin{equation*}
 		\frac{B_{m-2,m}}{B_{m-2,m-1}}=de^{(w_{m-1}-w_m)z^q}, 
 	\end{equation*}
 	where $d$ is a non zero constant. As $\rho(B_i)<q$ implies $\rho(B_{m-2,m})<q$, $\rho(B_{m-2,m-1})<q$. we get $(w_{m-1}-w_m)=0,$ which is not possible as all $w_i$'s are distincts.\\
 		\textbf{(ii)} $Q_{m-1}\equiv 0$ and $T_{m-1}\not\equiv0$.\\
 		By the similar argument as in  subcase B(ii), we have 
 		\begin{equation}\label{15}
 			f^nf'=c_1B_1e^{w_1z^q}+c_2B_2e^{w_2z^q}+...+c_{m-1}B_{m-1}e^{w_{m-1}z^q},			\end{equation} 
 			where $c_i$'s are constants and $c_{m-1}$ is nonzero. From lemma \ref{3p} it follows that $c_1=c_2=...=c_{m-2}=0$, otherwise equation \eqref{15} has no meromorphic solution. we get $f^nf'=c_{m-1}B_{m-1}e^{w_{m-1}z^q}$ and hence,
 			\begin{align*}
 				P=P(z,f,f',...,f^{(t)})=&\sum_{\gamma}\alpha_{\gamma}(z)K_{\gamma}(z)e^{\frac{\gamma w_{m-1}z^q}{n+1}}\\&= B_1e^{w_1z^q}+...+(1-c_{m-1})B_{m-1}e^{w_{m-1}z^q}+B_me^{w_mz^q}.
 			\end{align*}
 			Applying Borel's lemma \ref{imple1} in the above equation  we  get $B_1=...=B_{m-2}=0$, a contradiction to the assumption.\\
 			\textbf{(iii)} $Q_{m-1}\not \equiv 0$ and $T_{m-1}\equiv0$.\\
 			Similarly as in the proof of subcase B(iii), we have 
 			\begin{equation*}
 				P=d_1B_1e^{w_1z^q}+d_2B_2e^{w_2z^q}+...+d_{m-1}e^{w_{m-1}z^q},
 			\end{equation*}
 			where $d_i$ $(i=1,2,...,m-1)$ are constants with $d_{m-1}\neq0$. With the help of equation \eqref{eq:main}, we have 
 			\begin{equation}\label{15}
 				f^nf'=(1-d_1)B_1e^{w_1z^q}+...+(1-d_{m-1})B_{m-1}e^{w_{m-1}z^q}+B_me^{w_mz^q}.
 			\end{equation}
 			By Lemma \ref{3p}, the equation \eqref{15} cannot have meromorphic solutions unless $d_1=d_2=\cdots=d_{m-1}=1$.Therfore, $f^nf'=B_me^{w_mz^q}$. Now putting $f$ in equation \eqref{eq:main} and in $P(z)$, we have
 			\begin{align*}
 				P=P(z,f,f',...,f^{(t)})=&\sum_{\gamma}\alpha_{\gamma}(z)K_{\gamma}(z)e^{\frac{\gamma w_{m}z^q}{n+1}}\\&=B_1e^{w_1z^q}+...+B_{m-1}e^{w_{m-1}z^q}.
  			\end{align*}
 			In similar manner by Borel's lemma \ref{imple1}, we get $B_1=B_2=...=B_{m-1}=0$, which yields a contradiction.\\
 			\textbf{(iv)} $Q_{m-1}\not \equiv 0$ and $T_{m-1}\not\equiv0$.
 			\\Differentiating equation \eqref{14}, we have
 			\begin{equation}\label{16}
 				Q_{m-1}'+T_{m-1}'=(B_{m-1,m}'+B_{m-1,m}w_mqz^{q-1})e^{w_mz^q},
 			\end{equation}
 			eliminating $e^{w_mz^q}$ from equations \eqref{14} and \eqref{16}, we have
 			\begin{align*}
	&B_{m-1,m}Q'_{m-1}-B'_{m-1,m}+B_{m-1,m}w_mqz^{q-1})Q_{m-1}
 		\\&=(B'_{m-1,m}+B_{m-1,m}w_mqz^{q-1})T_{m-1}-B_{m-1,m}T'_{m-1}.
 		\end{align*}
 		Assume that \begin{align}\label{A}
 			&B_{m-1,m}Q'_{m-1}-B'_{m-1,m}+B_{m-1,m}w_mqz^{q-1})Q_{m-1}
 			\\&\nonumber=(B'_{m-1,m}+B_{m-1,m}w_mqz^{q-1})T_{m-1}-B_{m-1,m}T'_{m-1}\equiv0.
 		\end{align}
 		From the definition of $B_{m-1,m}$ and above equation, we get 
 		\begin{equation*}
 			f^nf'=t_1B_1e^{w_1z^q}+...+t_mB_me^{w_mz^q},
 		\end{equation*}
 		where $t_i$'s are constants and $t_m\neq 0$. Lemma \ref{3p} concludes $t_1 = t_2=...=t_{m-1}=0$, hence $f^nf'=t_mB_me^{w_mz^q}$,  similarly we have 
 		\begin{align*}
 			P=P(z,f,f',...,f^{(t)})=&\sum_{\gamma}\alpha_{\gamma}(z)K_{\gamma}(z)e^{\frac{\gamma w_{m}z^q}{n+1}}\\&=B_1e^{w_1z^q}+...+B_{m-1}e^{w_{m-1}z^q}+(1-t_m)B_me^{w_mz^q},
 		\end{align*}
 		and again Borel's lemma \ref{imple1} implies $B_1=B_2=...=B_{m-1}\equiv0$, which is a contradiction to the assumption.
 		Now, assume 
 		 \begin{align*}
 			&B_{m-1,m}Q'_{m-1}-(B'_{m-1,m}+B_{m-1,m}w_mqz^{q-1})Q_{m-1}
 			\\&=(B'_{m-1,m}+B_{m-1,m}w_mqz^{q-1})T_{m-1}-B_{m-1,m}T'_{m-1}\not\equiv0.
 			\end{align*}
 		Using lemma \ref{imple10}, the right hand side of  equation \ref{A} is a differential polynomial in $f$ of degree atmost $n-m-1$ and left hand side is of degree $n+1$ with $n\geq m+2$ and contains a  factor $f^{n-m}$. Hence,
 		\begin{equation*}
 			f^{n-m}R_{m-1}=B_{m-1,m}Q'_{m-1}-(B'_{m-1,m}+B_{m-1,m}w_mqz^{q-1})Q_{m-1}.
 		\end{equation*} From Clunie's lemma \ref{imple 4}, it follows that
 		\begin{equation*}
 			m(r,R_{m-1})=S(r,f) \; \text{and} \; m(r,fR_{m-1})=S(r,f).
 		\end{equation*}
 			Thus, $T(r,f)=S(r,f)$, which is impossible. This contradiction completes the proof for $B_0(z) \equiv0$.\\
 			Now consider the case $B_0(z)\not \equiv0$. \\
 			Define $P_1(z):=P(z,f,f',...,f^{(t)})-B_0$. Then $P_1(z)$ is a differential polynomial in $f$ of degree $d_1\leq n-m-1$. Consequently, equation \eqref{eq:main} can be rewritten as
 			 \begin{equation*}
 				f^nf'+P_1=B_1e^{w_1z^q}+...+B_me^{w_mz^q}.
 			\end{equation*} 
 			Applying  the same arguments as in the case for $B_0(z)\equiv0$, we obtain the desired conclusion.
 		 		 	 		\end{proof}

\end{document}